\documentclass{article}
\usepackage{amsmath,amsfonts,amsthm,amssymb,epsfig}

\newtheorem{theorem}{Theorem}
\newtheorem{lemma}[theorem]{Lemma}
\newtheorem{corollary}[theorem]{Corollary}
\theoremstyle{definition}
\newtheorem{definition}[theorem]{Definition}
\newtheorem{algorithm}[theorem]{Algorithm}
\newtheorem{remark}[theorem]{Remark}
\newtheorem*{ack}{Acknowledgements}

\newcommand{\eps}{\varepsilon}
\newcommand{\cc}{\mathrm{c}}
\newcommand{\cB}{\mathcal{B}}
\newcommand{\cF}{\mathcal{F}}
\renewcommand{\Pr}{\mathbb{P}}
\renewcommand{\le}{\leqslant}
\renewcommand{\ge}{\geqslant}
\DeclareMathOperator{\aut}{aut}

\newcommand\hH{\widetilde{H}}
\newcommand\hsH{\widehat{H}}
\newcommand\tQ{\widehat{Q}}

\begin{document}

\title{Random cliques in random graphs and sharp thresholds for $F$-factors}
\author{Oliver Riordan%
\thanks{Mathematical Institute, University of Oxford, Radcliffe Observatory Quarter, Woodstock Road, Oxford OX2 6GG, UK.
E-mail: {\tt riordan@maths.ox.ac.uk}.}}
\date{June 9, 2022}

\maketitle

\begin{abstract}
  We show that for each $r\ge 4$,
  in a density range extending up to, and slightly beyond, the threshold for a $K_r$-factor,
  the copies of $K_r$ in the random graph $G(n,p)$ are randomly distributed, in the (one-sided) sense
  that the hypergraph that they form contains a copy of a binomial random hypergraph with
almost exactly the right density. Thus Jeff Kahn's recent asymptotically sharp bound for the
threshold in Shamir's hypergraph matching problem implies
a corresponding bound for the threshold for $G(n,p)$ to contain a $K_r$-factor. The case $r=3$ is more difficult, and has been settled by Annika Heckel.
We also prove a corresponding result for $K_r^{(t)}$-factors in random $t$-uniform hypergraphs, as well as (in some cases weaker) generalizations replacing $K_r$ by certain other (hyper)graphs.
\end{abstract}

\section{Introduction and results}

For $r\ge 2$, $n\ge 1$ and $0\le p\le 1$, let $H_r(n,p)$ be the random hypergraph with
vertex set $[n]=\{1,2,\ldots,n\}$ in which each of the $\binom{n}{r}$ possible hyperedges
is present independently with probability $p$. Let $G(n,p)=H_2(n,p)$ be the usual
binomial (or Erd\H os--R\'enyi) random graph. An event (formally a sequence of events
indexed by $n$) holds \emph{with high probability} or \emph{whp} if its probability
tends to $1$ as $n\to\infty$.

According to Erd\H os~\cite{Erdos},
in 1979 Shamir posed the following extremely natural question (for $r=3$): how large should
$p=p(n)$ be for $H_r(n,p)$ to whp contain a perfect matching, i.e., a set of disjoint hyperedges
covering all vertices? (Of course, we assume implicitly that $r|n$.)
A related question is: given a fixed graph $F$, how large must $p$ be for 
$G(n,p)$ to whp contain an $F$-factor, i.e., a set of vertex-disjoint copies of $F$ covering
all vertices of $G$? This question was posed, and a conjecture for the answer was given,
by Ruci\'nski~\cite{Rucinski} and by Alon and Yuster~\cite{AY}.

After a number of partial results
on one or both of these questions, including~\cite{SS,Rucinski,AY,Kriv97,Kim2003}, they were solved up
to a constant factor in $p$ at the same time, and by the same method, 
in the seminal paper of Johansson, Kahn and Vu~\cite{JKV}. Although they
were solved together, one question \emph{appears} to be much simpler, and one might
wonder whether one question can be reduced to the other.
The main aim of this paper is to show that the answer is yes, in the following sense.

\begin{theorem}\label{th1}
Let $r\ge 4$ be given. There exists some $\eps=\eps(r)>0$ such that, for any $p=p(n)\le n^{-2/r+\eps}$,
the following holds. For some $\pi=\pi(n)\sim p^{\binom{r}{2}}$,
we may couple the random graph $G=G(n,p)$ with the random hypergraph $H=H_r(n,\pi)$ so that,
whp, for every hyperedge in $H$ there is a copy of $K_r$ in $G$ with the same vertex set.
\end{theorem}

Note that $\pi$ is (asymptotically) `what it should be', i.e., the probability that $r$ given vertices form a clique
in $G$. Thus almost all $K_r$s in $G$ will correspond to hyperedges in $H$, and the result says, roughly speaking,
that the $K_r$s in $G$ are distributed randomly. The precise statement involves a one-way bound:
we cannot expect to find a corresponding hyperedge of $H$ for every $K_r$ in $G$, since in $G$ we expect
to find on the order of $n^{2r-2}p^{2\binom{r}{2}-1}$ pairs of $K_r$s sharing two vertices which, when $p\to 0$,
is much larger than the expected number of pairs of hyperedges of $H$ sharing two vertices.

Theorem~\ref{th1} reduces certain questions
about the set of cliques in $G(n,p)$, whose distribution is very complicated due to the dependence between overlapping
cliques, to corresponding questions about $H_r(n,\pi)$, a much simpler random object.
This applies in particular to the $K_r$-factor question above,
relating it (one-way, but the other bound is easy)
to the threshold for a matching in $H_r(n,\pi)$ (Shamir's problem).
Indeed, the arguments of Johansson, Kahn and Vu~\cite{JKV} simplify considerably
when considering Shamir's problem (see the presentation in Chapter 13 of~\cite{FK}, for example).
This simpler version of their argument plus
Theorem~\ref{th1} gives an alternative proof of their $K_r$-factor result.

More significantly, Jeff Kahn~\cite{Kahn1} recently proved the following asymptotically sharp result for Shamir's problem.
\begin{theorem}[\cite{Kahn1}, Theorem 1.4]\label{Kthm}
Fix $r\ge 3$ and let 
\[
 \pi_0(n) = \binom{n-1}{r-1}^{-1}\log n \sim  (r-1)! n^{-r+1}\log n.
\]
For $\eps>0$ constant, whp $H_r(n,(1+\eps)\pi_0)$ contains a complete matching.\qed
\end{theorem}
Combined, this and Theorem~\ref{th1} have the following immediate corollary.
\begin{corollary}
Fix $r\ge 4$. Let $p_0(n)=\pi_0(n)^{1/\binom{r}{2}}$, where $\pi_0$ is as in Theorem~\ref{Kthm}. Then $p_0$ is a sharp threshold for $G(n,p)$ to contain a $K_r$-factor.
\end{corollary}
\begin{proof}
Fix $\eps>0$. For $p=(1-\eps)p_0$ it is well known and easy to check that whp there is at least one vertex of $G(n,p)$ not contained in a copy of $K_r$, so there is no $K_r$-factor. For $p=(1+\eps)p_0$ consider the coupling guaranteed (whp) by Theorem~\ref{th1}, noting that the $\pi$ we obtain satisfies $\pi\sim (1+\eps)^{\binom{r}{2}}\pi_0$. In particular (for large $n$) $\pi\ge (1+\eta)\pi_0$ for some constant $\eta>0$, so by Theorem~\ref{Kthm} whp $H_r(n,\pi)$ contains a complete matching. When the coupling succeeds (as it does whp), this implies the existence of a $K_r$-factor in $G(n,p)$.
\end{proof}

\begin{remark}
With an eye to even sharper results, one might wonder what the error term
in Theorem~\ref{th1} is; the proof below gives a bound $\pi=(1-n^{-\delta})p^{\binom{r}{2}}$
for some constant $\delta=\delta(r)>0$. This could presumably be improved, but it seems too much to hope that the recent hitting time result of Kahn~\cite{Kahn2} could be transferred from Shamir's problem to the $K_r$-factor problem using the methods of this paper.
\end{remark}

The omission of the case $r=3$ may appear strange. This case \emph{seems} much simpler, but, surprisingly, there
is an obstacle to the proof in this particular case. Annika Heckel~\cite{Heckel} managed to overcome this, proving an analogue of Theorem~\ref{th1} for $r=3$. Despite this, we will state and prove a weaker form of this result in Section~\ref{sec2}, since the proof illustrates in a simple context
a `thinning' technique used in Section~\ref{sec_bal}, which may perhaps be useful elsewhere.

\subsection{Extensions}

Although our main focus is the graph case, we also prove corresponding results for hypergraphs. For $r>t\ge 2$, let $K_r^{(t)}$ denote the complete $t$-uniform hypergraph on $r$ vertices.

\begin{theorem}\label{thhyp}
Let $r>t$ be given with $t\ge 2$ and $r\ge 4$. There exists some $\eps=\eps(r,t)>0$ such that, for any $p=p(n)\le n^{-(r-1)/\binom{r}{t}+\eps}$,
the following holds. For some $\pi=\pi(n)\sim p^{\binom{r}{t}}$,
we may couple the random hypergraph $G=H_t(n,p)$ with the random hypergraph $H=H_r(n,\pi)$ so that,
whp, for every hyperedge in $H$ there is a copy of $K_r^{(t)}$ in $G$ with the same vertex set.
\end{theorem}
Of course, the $t=2$ case of Theorem~\ref{thhyp} is simply Theorem~\ref{th1}. We have stated the graph case separately as it seems most interesting, and (to the author) less confusing.

Once again, combined with Kahn's Theorem~\ref{Kthm}, this has the following corollary, giving the `correct' asymptotic threshold for a $K_r^{(t)}$-factor in $H_t(n,p)$.
\begin{corollary}
Fix $r>t$ with $t\ge 2$ and $r\ge 4$, and define $\pi_0(n)$ as in Theorem~\ref{Kthm}. Then $p(n)=\pi_0(n)^{1/\binom{r}{t}}$ is a sharp threshold for $H_t(n,p)$ to contain a $K_r^{(t)}$-factor.\qed
\end{corollary}

We also prove an extension to certain non-complete graphs or hypergraphs.

\begin{definition}\label{dd1}
If $F$ is a (hyper)graph with at least two vertices, let
\[
 d_1(F) = e(F)/(|F|-1)
\]
be the \emph{$1$-density} of $F$. We say that $F$ is \emph{$1$-balanced} if
$d_1(F')\le d_1(F)$ for all sub(hyper)graphs $F'\subseteq F$ with at least two vertices, and \emph{strictly $1$-balanced}
if this inequality is strict for all such $F'\subsetneq F$.
\end{definition}
$1$-balanced is the natural notion
of balanced when studying $F$-factors, since the expected number of copies of $F$ in $G(n,p)$ (or $H_t(n,p)$)
containing a given vertex is of order $n^{|F|-1}p^{e(F)}$. The term \emph{balanced} is used in~\cite{JKV}, but 
we avoid this since it means too many different things in different contexts.

Theorems~\ref{th1} and~\ref{thhyp} can be generalized, at least to some extent, to certain $1$-balanced
(hyper)graphs $F$.
Since the statements are a little technical, we postpone them to Section~\ref{sec_bal}, stating
here two consequences concerning $F$-factors, one weak but relatively general, and one strong but with extra conditions on $F$.

\begin{theorem}\label{thbal}
Let $F$ be a $1$-balanced $t$-uniform hypergraph, where $t\ge 2$. There is some constant
$a=a_F$ such that if $p=p(n)\ge (\log n)^a n^{-1/d_1(F)}$, and $|F|$ divides $n$, then whp $H_t(n,p)$
contains an $F$-factor.
\end{theorem}
Note that this result is tight up to the log factor. When $F$ is strictly $1$-balanced, then 
Johansson, Kahn and Vu~\cite{JKV} gave a sharper result (finding the threshold up to a constant
factor), but for other graphs they gave a result with an $n^{o(1)}$ error term although, as pointed
out by a referee, with some care their method would also give a power of log as the error term.
Gerke and McDowell~\cite{GMcD} gave a sharp (up to constants) result for a certain class of unbalanced
graphs (which they call `nonvertex-balanced').
Theorem~\ref{thbal} extends to the multipartite multigraph setting of~\cite{GMcD}.

Finally, we turn to asymptotically sharp results. Here we need some further definitions. We say that a hypergraph $F$ is \emph{$k$-connected} if it has at least $k+1$ vertices and has no cutset $S$ of size at most $k-1$, where $S\subset V(F)$ is a \emph{cutset} if we may write $F=F_1\cup F_2$ where $V(F_1)\cap V(F_2)=S$ and neither $V(F_1)$ nor $V(F_2)$ is contained in $S$. For graphs, this is exactly the usual notion of $k$-connectivity.

\begin{definition}\label{defnice}
A $t$-uniform hypergraph $F$ is \emph{nice} if (i) $F$ is strictly 1-balanced, (ii) $F$ is $3$-connected,
and (iii) either $t\ge 3$, or $t=2$ and $F$ cannot be transformed into an isomorphic graph by adding one edge and deleting one edge.
\end{definition}
Note that the restriction (iii) is only needed in the graph case, and is satisfied by any regular graph. An example of a graph $F$ satisfying (i) and (ii) but not (iii) is $K_5$ with an edge deleted. Nice hypergraphs are the class for which the transfer argument in this paper gives a sharp result for the $F$-factor threshold.

\begin{theorem}\label{thnice}
Let $F$ be a fixed nice $t$-uniform hypergraph with $r$ vertices and $s$ edges. Then
\[
 p_0(n)=\bigl((\aut(F)/r) n^{-r+1}\log n\bigr)^{1/s}
\]
is a sharp threshold for $H_t(n,p)$ to contain an $F$-factor.
\end{theorem}

The rest of the paper is organized as follows. In Section~\ref{sec_prelim} we give some further
definitions and some preparatory lemmas. Theorems~\ref{th1} and~\ref{thhyp} are proved in Section~\ref{sec1}. In Section~\ref{sec2} we illustrate a `thinning technique', proving a weaker form of Heckel's $r=3$ result. Generalizations of (in some cases a weaker form of) Theorem~\ref{thhyp}
to certain (hyper)graphs other than $K_r$ are stated and proved in Section~\ref{sec_bal},
and Theorems~\ref{thbal} and~\ref{thnice} are proved there. Finally,
we finish with a brief discussion of open questions in Section~\ref{sec_q}.

\section{Preliminaries}\label{sec_prelim}

Fix $r>t\ge 2$. In this and the next section, we work simultaneously with $r$-uniform hypergraphs, and with graphs (for Theorem~\ref{th1}) or $t$-uniform hypergraphs (for Theorem~\ref{thhyp}). We will refer to the latter as `graphs' (when $t=2$) or as `$t$-graphs'; we use the term `hypergraph' to mean an $r$-uniform hypergraph. On a first reading, the reader may wish to focus on the the case $t=2$, so $t$-graphs become simply graphs.

Given a hypergraph $H$, we write
$|H|$, $e(H)$ and $c(H)$ for the number of vertices, hyperedges\footnote{Since we consider ($t$-)graphs and hypergraphs simultaneously, we will try to distinguish ($t$-)graph edges from hyperedges}, and components of $H$, and
\[
 n(H) = (r-1) e(H) + c(H) - |H|
\]
for the \emph{nullity} of $H$, which is simply the usual (graph) nullity of any multigraph obtained
from $H$ by replacing each hyperedge by a tree with the same vertex set. We will need
this definition only in the connected case. Note that $n(H)\ge 0$, and (for connected $H$),
$n(H)=0$ if and only if $H$ is a \emph{tree}, i.e., can be built by starting with a single vertex,
and at each step adding a new hyperedge meeting the existing vertex set in exactly one vertex.

A connected hypergraph $H$ is \emph{unicyclic} if $n(H)=1$ and \emph{complex} if $n(H)\ge 2$. Thus, 
for example, any connected hypergraph containing two hyperedges that share three or more vertices is complex.

\begin{definition}
By an \emph{avoidable configuration} we mean 
a connected, complex hypergraph with at most $2^{r+1}$ hyperedges.\footnote{The constant $2^{r+1}$ here is somewhat arbitrary, chosen large enough to (easily) cover all applications of this concept through the paper.}
\end{definition}

The motivation for this definition  is the fact (proved in a moment) that
such configurations will (whp) not appear in random hypergraphs of the density we consider.
Indeed, roughly speaking, these random hypergraphs are locally tree-like around most vertices,
with some unicyclic exceptions.
Globally, they can be far from unicyclic. We record this simple observation as a lemma for ease
of reference, and give the trivial proof for completeness.

\begin{lemma}\label{l1}
For each fixed $r\ge 2$ there is an $\eps>0$ with the following property. 
If $H=H_r(n,\pi)$ with $\pi=\pi(n)\le n^{-(r-1)+\eps}$, then whp $H$
contains no avoidable configurations.
\end{lemma}
\begin{proof}
Fix $r\ge 2$. Any avoidable configuration is a connected hypergraph of bounded size, so up to isomorphism
there are $O(1)$ of them. Let $C$ be any avoidable configuration.
Then the expected number of copies of $C$ in $H$
is $\Theta(n^{|C|}\pi^{e(C)}) \le n^{|C|-(r-1)e(C)+O(\eps)}$.
But $C$ is complex and connected, so $(r-1)e(C)-|C|=n(C)-1\ge 1$, so this expectation is at most
$n^{-1+O(\eps)}\le n^{-0.99}=o(1)$ if $\eps$ is sufficiently small.
\end{proof}

The next (deterministic) lemma shows that if, when we replace each hyperedge of a hypergraph $H$ by a copy of $K_r^{(t)}$,
there is an `extra' copy of $K_r^{(t)}$ (one that does not correspond to a hyperedge in $H$),
then $H$ must contain an avoidable configuration. It is here that the cases $t=2$ and $t\ge 3$ differ most.

\begin{lemma}\label{l2}
Suppose that $r>t\ge 2$ and that $r \ge 4$. Let $H$ be an $r$-uniform hypergraph, and let $G$ be the $t$-graph obtained
from $H$ by replacing each
hyperedge by a copy of $K_r^{(t)}$ (merging any multiple edges). If $G$ contains a copy $F$ of $K_r^{(t)}$ on
a set of $r$ vertices which is not a hyperedge in $H$, then $H$ contains an avoidable configuration.
\end{lemma}
\begin{proof}
By assumption there are hyperedges $h_1,\ldots,h_k$ of $H$ such that the union of the
corresponding copies $F_1,\ldots,F_k$ of $K_r^{(t)}$ includes $F$, the complete $t$-graph on a set $h\notin E(H)$
of $r$ vertices. Clearly, we may assume that each $F_i$ shares at least $t$ vertices with $F$
and (removing `redundant' $F_i$ that contribute no `new' ($t$-)edges to $F$ not covered by earlier $F_j$) that $k\le \binom{r}{t}$.
Let $C$ be the hypergraph with hyperedges $h_1,\ldots,h_k$, with vertex set $\bigcup_{i=1}^k h_i$.
Let $C^+=C+h$ be the hypergraph formed from $C$ by adding $h$ as a hyperedge. (Its vertices are
all included already.)

Certainly, $C$ is connected: otherwise, its components would partition $V(h)=V(F)$, and those ($t$-)edges of $F$ not
contained within a part of this partition would not be covered by $\bigcup F_i$. Also, $e(C)=k\le\binom{r}{t}\le 2^r$. So it remains
only to show that $n(C)\ge 2$; then $C\subset H$ is the required avoidable configuration.

Let $s_i=|F_i\cap F|=|h_i\cap h|$ be the number of vertices shared by $h_i$ and $h$.
Then, considering adding the hyperedges in the order $h,h_1,\ldots,h_k$, we have
\[
 n(C^+)\ge \sum_{i} (s_i-1).
\]
On the other hand, considering the ($t$-)edges of $F$ covered by each $F_i$,
\begin{multline}
 \binom{r}{t} \le \sum_i \binom{s_i}{t} \le \max\left\{\frac{\binom{s_i}{t}}{s_i-1}\right\} \sum_i(s_i-1)  
\\
 \le \max\left\{\frac{\binom{s_i}{t}}{s_i-1}\right\} n(C^+)
 \le \frac{\binom{r-1}{t}}{r-2} n(C^+), \label{bd}
\end{multline}
since none of the $F_i$ is equal to $F$ (so $s_i\le r-1$), and $\binom{x}{t}/(x-1)=x(x-2)(x-3)\cdots(x-t+1)/t!$ is strictly increasing in $x\ge t$.
It follows that
\[
 n(C^+)\ge (r-2)\frac{\binom{r}{t}}{\binom{r-1}{t}} = \frac{r(r-2)}{r-t} \ge r,
\]
with equality only if equality holds throughout \eqref{bd} and $t=2$. But then (in the equality case) all $s_i$
must be equal to $r-1$, so
any two $F_i$ overlap within $F$ in at least $r-2\ge 2=t$ vertices, so the first inequality in \eqref{bd} is strict.
Hence $n(C^+)\ge r+1$, so $n(C)=n(C^+)-(r-1)\ge 2$, as required.
\end{proof}

\begin{remark}\label{RA3}
The conclusion of Lemma~\ref{l2} does not hold when $r=3$ and $t=2$. Following through the proof, the condition $r\ge 4$
was only used in the second-last sentence. Thus we see that for $r=3$, $t=2$ there is a single exceptional
configuration: three triples with each pair meeting in a (distinct) vertex; we later refer
to this as a `clean $3$-cycle'; see the first diagram in Figure~\ref{fig_cycles}.
\end{remark}

\begin{figure}[htb]
 \begin{center}
\includegraphics[width=4.5in]{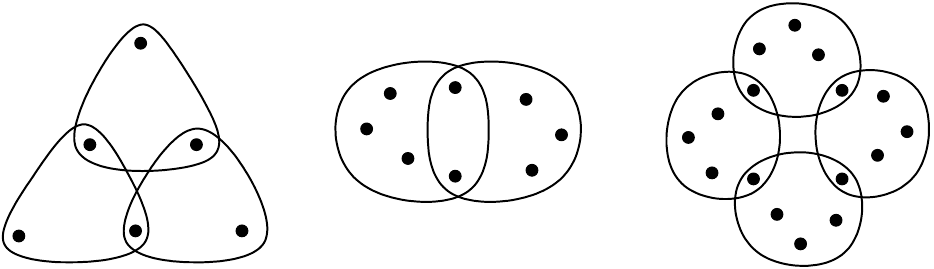}
 \end{center}
 \caption{Clean $k$-cycles in $r$-uniform hypergraphs for $(k,r)=(3,3)$, $(2,5)$ and $(4,5)$. At least when $k\ge 3$, such cycles are also known as `loose $k(r-1)$-cycles'.}\label{fig_cycles}
\end{figure}

\section{Proof of Theorems~\ref{th1} and~\ref{thhyp}}\label{sec1}

In this section we prove Theorem~\ref{thhyp} and thus Theorem~\ref{th1}, which is the special case $t=2$.
We attempt to optimize the terminology for the case $t=2$, speaking of ($t$-)edges or just edges in
the $t$-uniform hypergraph $G$, and hyperedges for the $r$-uniform $H$.

The overall strategy is similar to one employed by Bollob\'as and the author in~\cite{rmax}.
Only at one point will we need to assume $r\ge 4$, so most of the time we assume only that $r>t\ge 2$.
In essence, the idea is to
test for the presence of each possible $K_r^{(t)}$ in $G=H_t(n,p)$ (thus $G=G(n,p)$ when $t=2$) one-by-one, each time only observing whether
the $K_r^{(t)}$ is present or not, not which edges are missing in the latter case. It suffices to show that,
at least on a global event of high probability (meaning, as usual, probability $1-o(1)$ as $n\to\infty$),
the conditional probability that a certain test succeeds
given the history is at least~$\pi$.

There will be some complications. A minor one is that we would like to keep control
of the copies of $K_r^{(t)}$ `found so far' by using $H=H_r(n,\pi)$ rather than $G$, since we don't want to find too 
many copies. The solution to this is simple: if the conditional probability of a certain test succeeding
given the history is $\pi'>\pi$, then we toss a coin independent of $G$ (and of all other coins),
only actually testing $G$ for the copy of $K_r^{(t)}$ with (conditional) probability $\pi/\pi'$.

Another complication is that it will happen with significant probability that some tests
that we would like to carry out have conditional probability less than $\pi$ of succeeding. Roughly speaking,
as long as this happens $o(1/\pi)$ times, we are ok. More precisely, each time this happens we toss a $\pi$-biased coin
to determine whether the relevant hyperedge is present in $H$, and if so, our coupling fails.
We will show that the coupling succeeds on a global event of high probability.

Turning to the details, fix $r>t\ge 2$. Let $M=\binom{n}{r}$,
and let $E_1,\ldots,E_M$ denote the ($t$-)edge-sets of all possible copies
of $K_r^{(t)}$ in $G=H_t(n,p)$. Let $A_i$ be the event that $E_i\subset E(G)$, i.e., that the $i$th copy
is present. As outlined above, our algorithm proceeds as follows, revealing some information
about $G=H_t(n,p)$ while simultaneously constructing $H=H_r(n,\pi)$.

\begin{algorithm}\label{alg}
For each $j$ from $1$ to $M$:

First calculate $\pi_j$, the conditional probability of the event $A_j$ given all information revealed so far.

If $\pi_j\ge \pi$, then toss a coin with heads probability $\pi/\pi_j$. If it lands heads, test whether the event $A_j$
holds. If so, declare the hyperedge $h_j$ corresponding to $E_j$ to be present in $H$. If not, or if the coin lands tails, 
declare $h_j$ to be absent.

If $\pi_j<\pi$, then toss a coin with heads probability $\pi$, and simply declare $h_j$ to be present in $H$ if this coin lands heads. If this happens, our coupling has failed.
\end{algorithm}

At the end, the hypergraph $H$ we have constructed clearly has the correct distribution for $H_r(n,\pi)$,
so it remains only to show that the probability that the coupling fails is $o(1)$.

Suppose that we have reached step $j$ of the algorithm; our aim is to bound $\pi_j$. In the previous
steps, we have `tested' whether certain (not necessarily all) of the events $A_1,\ldots,A_{j-1}$
hold, in each case receiving the answer `yes' or `no'. Suppressing the dependence on $j$ in the notation,
let $Y$ and $N$ denote the corresponding
(random) subsets of $[j-1]$. Then, from the form of the algorithm, the information
about $G$ revealed so far is precisely that every event $A_i$, $i\in Y$, holds, and none of the events $A_i$, $i\in N$,
holds.

\begin{figure}[htb]
 \begin{center}
\includegraphics[width=4.2in]{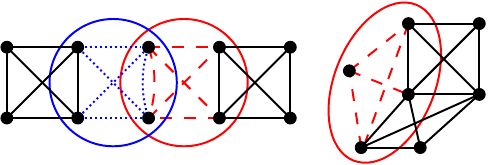}
 \end{center}
 \caption{A possible state of the algorithm. The black $K_4$s (not circled to avoid clutter) have been found to be present; $R$ consists of the black edges. Each red circled $K_4$ has been found to be absent, meaning within the circle, at least one of the red dashed edges is absent; these sets of edges form the $E_i'$. We are about to test for the blue circled $K_4$, i.e., the set $E_j'$ of blue dotted edges. With the black edges fixed as present, this is testing for an up-set conditional on a number of down-sets. In the key estimate (bounding $Q_j$ defined in \eqref{Qdef}) $j$ (the index of the blue $K_4$) is fixed, and we sum over $i$ such that $E_i'$ shares at least one edge with $E_j'$. We group the terms according to the pattern formed by $R\cup E_j'$ within the \emph{red} $K_4$ corresponding to $E_i$.}\label{fig_alg}
\end{figure}

Let $R=\bigcup_{i\in Y} E_i$ be the set of ($t$-)edges `revealed' so far. For $i\le j$ let $E_i'=E_i\setminus R$.
Then what we know about $G=H_t(n,p)$ is precisely that all edges in $R$ are present, and none of the
\emph{sets} $E_i'$, $i\in N$, of edges is present; see Figure~\ref{fig_alg}.
Working in the random ($t$-)graph $G'$ in which
each ($t$-)edge outside $R$ is present independently with probability $p$, and
writing $A_i'$ for the event $E_i'\subset E(G')$, we have
\[
 \pi_j = \Pr\left( E_j' \subset E(G') \Bigm| \bigcap_{i\in N} \left\{E_i'\not\subset E(G') \right\} \right)
 = \Pr\left(A_j' \Bigm| \bigcap_{i\in N} (A_i')^\cc \right).
\]
To estimate this probability we follow a standard strategy from the proof of Janson's inequality,
using a variation suggested by Lutz Warnke (see~\cite{RW-Janson}).
As usual, the starting point
is to consider which events $A_i'$ are independent of $A_j'$.
In particular, let
\[
 D_0 = \bigcap_{i\in N\::\: E_i'\cap E_j'=\emptyset} (A_i')^\cc
\quad\text{and}\quad
 D_1 = \bigcap_{i\in N\::\: E_i'\cap E_j'\ne\emptyset} (A_i')^\cc.
\]
Then
\begin{multline*}
 \pi_j = \Pr(A_j'\mid D_0\cap D_1) = \frac{\Pr(A_j'\cap D_0\cap D_1)}{\Pr(D_0\cap D_1)}
 \ge \frac{\Pr(A_j'\cap D_0\cap D_1)}{\Pr(D_0)} \\
 = \Pr(A_j'\cap D_1 \mid D_0) = \Pr(A_j'\mid D_0)-\Pr(A_j'\cap D_1^\cc \mid D_0).
\end{multline*}
Now $D_0$ only involves the presence or absence (in $G'$) of edges not in $E_j'$,
so $A_j'$ and $D_0$ are independent. Also, $A_j'\cap D_1^\cc$ is an up-set, while $D_0$
is a down-set. Hence, by Harris's inequality, $\Pr(A_j'\cap D_1^\cc \mid D_0) \le \Pr(A_j'\cap D_1^\cc)$.
Thus,
\[
 \pi_j \ge \Pr(A_j') - \Pr(A_j'\cap D_1^\cc).
\]
Let
\[
 N_1 =N_{j,1} = \{i\in N: E_i' \cap E_j' \ne\emptyset \},
\]
so $D_1=\bigcap_{i\in N_1} (A_i')^\cc$ and hence $D_1^\cc=\bigcup_{i\in N_1} A_i'$.
Then, using the union bound, we have
\begin{equation}\label{prekey}
 \pi_j \ge \Pr(A_j') - \sum_{i\in N_1}\Pr(A_j'\cap A_i').
\end{equation}
Hence
\begin{equation}\label{key}
 \pi_j \ge p^{|E_j\setminus R|} - \sum_{i\in N_1} p^{|(E_j\cup E_i)\setminus R|}
 = p^{|E_j\setminus R|}(1-Q) \ge p^{e(K_r^{(t)})} (1-Q),
\end{equation}
where
\begin{equation}\label{Qdef}
 Q = Q_j = \sum_{i\in N_1} p^{|E_i\setminus (E_j\cup R)|},
\end{equation}
which is of course random (depending, via $N_1$ and $R$, on the information revealed so far).
To prove Theorem~\ref{thhyp} it suffices, roughly speaking, to show that almost always $Q=o(1)$.

\begin{proof}[Proof of Theorem~\ref{thhyp}]
For the moment, we consider any fixed $r>t\ge 2$.
We take $p\le n^{-(r-1)/\binom{r}{t}+o(1)}$ for notational simplicity; it should be clear from the proof that
follows that the arguments carry through when $p \le n^{-(r-1)/\binom{r}{t}+\eps}$ as long as $\eps$ is sufficiently
small, meaning at most a certain positive constant depending on $r$ and $t$.

For this $p$ we have $\pi\le n^{-(r-1)+o(1)}$. Hence the expectation of the degree of a vertex of $H_r(n,\pi)$
is $\binom{n-1}{r-1}\pi\le n^{o(1)}$. Since the actual degree is binomial,
it follows by a Chernoff bound that there is some $\Delta=n^{o(1)}$ such
that whp every vertex of $H_r(n,\pi)$ has degree at most $\Delta/2^r$, say.
Let $\cB_1$ be the `bad' event that some vertex of $H$, the final version of the hypergraph constructed as we
run our algorithm, has degree more than $\Delta/2^r$, so $\Pr(\cB_1)=o(1)$.

Let $\cB_2$ be the `bad' event that $H$ contains an avoidable configuration, as defined
in Section~\ref{sec_prelim}. By Lemma~\ref{l1} we have $\Pr(\cB_2)=o(1)$.

Consider some $1\le j\le \binom{n}{r}$, which will remain fixed through the rest of the
argument. As outlined above, we condition on the result of steps $1,\ldots,j-1$ of our exploration;
we will show that if $\pi_j<\pi$ \emph{and} the
hyperedge corresponding to $E_j$ is present in $H$ (the only case where
the coupling fails), then $\cB_1\cup \cB_2$ holds.
The graph $R=R_j$ of `found' edges is a subgraph of the ($t$-)graph
formed by replacing each hyperedge of $H$ by a $K_r^{(t)}$. Hence $\Delta(R)\le \binom{r-1}{t-1}\Delta(H)$,
and we may assume that 
\begin{equation}\label{ass1}
 \Delta(R\cup E_j)\le \Delta = n^{o(1)},
\end{equation}
since otherwise $\cB_1$ holds.

Let us consider a particular $i\in N_1=N_{j,1}$, and its contribution to $Q=Q_j$.
Let $S$ be the ($t$-)graph with vertex set $V(E_i)$ in which we include a ($t$-)edge if it is
in $R\cup E_j$. Then the contribution is exactly $p^{e_i}$, where $e_i=|E_i\setminus E(S)|$. For example, in
the situation illustrated in Figure~\ref{fig_alg}, with $i$ corresponding to the red $K_4$ on the left,
$S$ consists of two edges, one from $R$ and one (the curved blue dotted edge) from $E_j'=E_j\setminus R$; thus $e_i=4$.

Crudely, $e_i$ is at least the number of edges of $E_i$ (the complete $t$-graph on $V(S)$)
not contained within the vertex set of any component of $S$. Suppose first that $S$ has at least two components,
and let their orders be $r_1,\ldots,r_{k+1}$, $k\ge 1$; this numbering will be convenient
in a moment. Note that $\sum r_\ell=r$. Note also that (by definition of $N_1$),
$E_j$ and $E_i$ intersect in at least one edge. Thus $S$ has at least one ($t$-)edge
and so at most $r-t+1$ components, i.e., $k\le r-t$.

Given the constraint $\sum r_\ell=r$, with each $r_\ell\ge 1$, by convexity the sum of $\binom{r_\ell}{t}$
is maximized when $r_1=\cdots=r_k=1$ and $r_{k+1}=r-k$.
Thus
\[
 e_i \ge \binom{r}{t} - \sum_{\ell=1}^{k+1}\binom{r_\ell}{t} \ge \binom{r}{t} - \binom{r-k}{t}.
\]
We next consider how many $i$ may lead to a configuration of this type, specifically, 
one where $S$ has $k+1$ components. Note that $S$ is formed of edges in $R\cup E_j$, a ($t$-)graph
of maximum degree at most $\Delta=n^{o(1)}$, and includes at least one vertex of a given set $V(E_j)$ of size $r=O(1)$.
It follows that there are at most
\begin{equation}\label{choices}
 rn^k((t-1)\Delta)^{r-k-1}=n^{k+o(1)}
\end{equation}
such choices: $r$ choices for an initial
vertex in $V(E_j$), then at most $n$ choices each time we start
a new component other than the first, and (crudely) at most $(t-1)\Delta$ choices for each subsequent vertex within a component.
Hence the contribution of such terms ($S$ having $k+1\ge 2$ components) to $Q$ is at most
\[
 \sum_{k=1}^{r-t} n^{k+o(1)}p^{\binom{r}{t}-\binom{r-k}{t}}.
\]
A fairly standard calculation shows that this is $o(1)$ (in fact, bounded by a small negative
power of $n$); indeed, the power of $n$
in a given term of the sum is at most $o(1)$ plus
\[
 k - \left(\binom{r}{t}-\binom{r-k}{t}\right) \frac{r-1}{\binom{r}{t}} = k - (r-1) + (r-1)\binom{r-k}{t}\binom{r}{t}^{-1},
\]
and this function of $k$ is strictly convex on $[0,r-t]$, zero for $k=0$ and negative for $k=r-t$; it follows that it is negative for $1\le k\le r-t$.

It remains only to treat the case where $S$ is connected. There are at most $r((t-1)\Delta)^{r-1}=n^{o(1)}$ terms of this
form (by the argument for \eqref{choices} with $k=0$), so the contribution from those with $e_i>0$ is at most $n^{o(1)}p
\le n^{-(r-1)/\binom{r}{t}+o(1)}=o(1)$.
This leaves the case where $e_i=0$, i.e., where $E_i\subset R\cup E_j$. Such $E_i$ contribute
exactly $1$, so we have shown that if $\cB_1$ does not hold, then
\begin{equation}\label{Qest}
 Q_j = o(1) + |D_j|,
\end{equation}
where
\[
 D_j = \{i\in N_{j,1}: E_i\subset R\cup E_j\}.
\]
In particular, when $\cB_1$ does not hold and $D_j=\emptyset$, we have $Q_j=o(1)$. Thus, recalling \eqref{key}, we may choose $\pi\sim p^{\binom{r}{t}}$ so that in such cases $\pi_j\ge \pi$, and our coupling cannot fail at such a step.

Let us call step $j$ \emph{dangerous} if $D_j\ne\emptyset$.
Note that in any such step we have $\pi_j=0$,
since if we do have $E_j\subset E(G)$, 
then $E_j\cup R\subset E(G)$ but (since $i\in N$) we have $E_i\not\subset E(G)$, giving a contradiction.
In a dangerous step, we toss a new $\pi$-probability coin to determine
whether the hyperedge $h_j$ corresponding to $E_j$ is present in $H$. If it is, we call step $j$ \emph{deadly}.
Our coupling fails if and only if there is some deadly step $j$.
To complete the proof it thus suffices to show that if any step is deadly, then $\cB_2$ holds.
If step $j$ is deadly, then every ($t$-)edge in $E_j\cup R\supset E_i$
lies within some hyperedge of $H$, but the hyperedge corresponding to $E_i$ is not present in $H$
(since $i\in N$).
In this case, using (only now) the condition $r\ge 4$, by Lemma~\ref{l2} $H$ contains an avoidable configuration,
i.e., $\cB_2$ holds. Thus, if our coupling fails, $\cB_1\cup \cB_2$ holds, an event of probability $o(1)$. This completes the proof of Theorem~\ref{thhyp}.
\end{proof}

As noted earlier, Theorem~\ref{th1} is a special case of Theorem~\ref{thhyp}.

\section{The triangle case}\label{sec2}

In this section we prove the following result, a weakening of the (missing) $r=3$ case of Theorem~\ref{th1}. The result itself is of no relevance, since it is superseded by Annika Heckel's stronger result~\cite{Heckel}, but the proof method may perhaps be. In particular, we will use the same idea in a more complicated context in Section~\ref{sec_bal}, and it seems easier to introduce it in the present simple case.

\begin{theorem}\label{th2}
There exists a constant $\eps>0$ such that, for any $p=p(n)\le n^{-2/3+\eps}$,
the following holds. Let $a<1/4$ be constant, and let $\pi=\pi(n) = a p^3$.
Then we may couple the random graph $G=G(n,p)$ with the random hypergraph $H=H_3(n,\pi)$ so that,
whp, for every hyperedge in $H$ there is a copy of $K_3$ in $G$ with the same vertex set.
\end{theorem}

The proof of Theorem~\ref{th1} given in the previous section `almost' works for $r=3$.
The only problem is the unique exception to Lemma~\ref{l2}, a `clean' hypergraph 3-cycle;
an $r$-uniform hypergraph is a \emph{clean $k$-cycle} if it can be formed from a graph $k$-cycle
by adding $r-2$ new vertices to each edge, with the added vertices all distinct; see Figure~\ref{fig_cycles}. (We extend the definition to $k=2$, when it simply means two hyperedges sharing exactly $2$
vertices.)

\begin{remark}
Simply by `skipping over' dangerous steps, for $p\le n^{-2/3+\eps}$
the proof of Theorem~\ref{th1}
shows the existence of a coupling between $G=G(n,p)$
and $H=H_3(n,\pi)$, $\pi\sim p^3$, so that whp for every hyperedge of $H$ which is not in a clean 3-cycle
(i.e., almost all of them) there is a corresponding triangle in $G$.
\end{remark}
Alternatively, as claimed in Theorem~\ref{th2}, we can avoid leaving out any hyperedges of $H$, at the cost
of decreasing its density $\pi$ by a constant factor.

\begin{proof}[Proof of Theorem~\ref{th2}]
We are given a constant $a<1/4$. Fix another constant $0<c<1$ such that 
\begin{equation}\label{ca}
 c(1-c) > a.
\end{equation}
(Of course $c=1/2$ always works; in Section~\ref{sec_bal} other choices will be useful.)
In the previous section we examined the random graph $G=G(n,p)$ according to Algorithm~\ref{alg},
checking copies of $F=K_r$ for their presence one-by-one. Here, in addition to the 
random variables corresponding to the edges of $G(n,p)$, we consider one $0$/$1$-random variable
$I_j$ for each of the $\binom{n}{3}$ possible copies $F_j$ of $K_3$,
with $\Pr(I_j=1)=c$.
We take the $I_j$ and the indicators of the presence of the edges in $G(n,p)$ to be independent.
We think of the $I_j$ as `thinning' the copies of $K_3$ in $G(n,p)$, selecting a random subset.
Note that $I_j$ should not be confused with the random variable describing the presence of the
corresponding hyperedge in $H$.

With $\pi=a p^3$, our aim will be to construct a random hypergraph $H$ with the distribution
of $H_3(n,\pi)$ so that for every hyperedge in $H$ there is a triangle in $G$ with $I_j=1$.
In other words, we try to embed (in the coupling as a subhypergraph sense) $H$ within the `thinned
triangle hypergraph' $H_3^-(G)$ having a hyperedge for each triangle in $G$ with $I_j=1$.
This clearly suffices. But how does making things (apparently) harder for ourselves in this way help?

We follow the proof of Theorem~\ref{th1} very closely. Consider the random (non-uniform) hypergraph $G^*$,
with edge set $E(G(n,p))\cup \{F_j:I_j=1\}$, i.e., an edge for each edge of $G=G(n,p)$, and a triple
for each $j$ such that $I_j=1$.
We follow the same algorithm as before, \emph{mutatis mutandis}\footnote{`Changing what must be changed', i.e., with the obvious modifications to the new setting.}, now examining $G^*$ rather than $G$.
At each step we check whether a given triangle $F_j$ is present `after thinning', i.e., whether
it is the case that $E_j\subset E(G)$ and $I_j=1$, where $E_j$ is the edge-set of $F_j$.
In other words, we test whether  $E_j^*\subset E(G^*)$, where $E_j^*$
consists of the edges $E_j$ together with one hyperedge corresponding to $F_j$; an individual
event of this form has probability $cp^3$. 
As before, we only record the overall yes/no answer, and write $\pi_j$ for the conditional probability
of this test succeeding given the history.
Because the (hyper)edges of $G^*$ are present
independently, the argument leading to \eqref{prekey} carries through \emph{exactly} as before,
but now with $A_j$ the event that $E_j^*\subset E(G^*)$, and with $R$ the set of (hyper)edges of $G^*$ found
so far. Noting that each triangle has its own `extra' hyperedge,
in place of \eqref{key} we thus obtain
\begin{equation}\label{key2}
 \pi_j \ge cp^{|E_j\setminus R|} - \sum_{i\in N_{j,1}} c^2p^{|(E_j\cup E_i)\setminus R|}
 = cp^{|E_j\setminus R|}(1-cQ) \ge cp^3 (1-cQ),
\end{equation}
where, as before,
\[
 Q = Q_j = \sum_{i\in N_{j,1}} p^{|E_i\setminus (E_j\cup R)|}.
\]
The key point is that the first term in \eqref{key2} contains one factor of $c$ (from the probability 
that $I_j=1$), while the second contains two, from the probability that $I_i=I_j=1$.

We estimate $Q_j$ exactly as before, leading to the bound \eqref{Qest}, valid whenever
$\cB_1$ does not hold. This time, let us call step $j$ \emph{dangerous} if there are 
two (or more) distinct $i,i'\in N_{j,1}$ such that $E_i$ and $E_{i'}$ are both contained in $E_j\cup R$.
If step $j$ is not dangerous, then from \eqref{Qest} we have $Q_j\le 1+o(1)$,
which with \eqref{key2} gives
\[
 \pi_j \ge (1+o(1))c(1-c)p^3 \ge a p^3 = \pi,
\]
for $n$ large enough, where in the second step we used \eqref{ca}. Hence
our coupling cannot fail at such a step.

As before, we call a dangerous step $j$ \emph{deadly} if the hyperedge (now a triple) corresponding
to $E_j$ is present in the random hypergraph $H=H_3(n,\pi)$ that we construct. Our coupling
fails only if such a step exists. As before, this implies that the simple graph $G(H)$
corresponding to $H$ contains a triangle with edge-set $E_i$, even though $H$ contains
no triple corresponding to this triangle. We may assume that $H$ contains no avoidable configuration
(otherwise $\cB_2$ holds). By Remark~\ref{RA3}, it follows that $H$ contains a clean $3$-cycle $H_1$
`sitting on' $E_i$. Similarly, $H$ contains a clean $3$-cycle $H_2$ sitting on $E_{i'}$.
Since $i,i'\in N_{j,1}$ we have that $E_i$ and $E_i'$ both intersect $E_j$ in at least one edge. 
Hence there is a vertex common to $E_i$ and $E_i'$. It follows that $H_1$ and $H_2$ share at least
one vertex. Since they are unicyclic and not identical, it easily follows that their
union is connected and complex, and (since it has at most $6\le 2^r$ hyperedges)
is hence an avoidable configuration. So $\cB_2$ does
hold after all. Thus we have again shown that if our coupling fails, $\cB_1\cup \cB_2$ holds,
an event of probability $o(1)$.
\end{proof}

\section{Extension to $1$-balanced graphs}\label{sec_bal}

In this section we state and prove an extension to certain $1$-balanced ($t$-uniform hyper)graphs $F$, considering
copies of $F$ in $G(n,p)$ ($H_t(n,p)$) rather than copies of $K_r$. Our main focus is the graph case $t=2$, but it turns out that the proof can easily be written to extend to $t\ge 3$ with no changes. Still, we attempt to optimize the notation for $t=2$, writing `($t$-)graph' or sometimes just `graph' for a $t$-uniform hypergraph, as in Sections~\ref{sec_prelim} and~\ref{sec1}.

We shall write $r$ for $|F|$ and $s$ for
$e(F)$ throughout. Thus, recalling Definition~\ref{dd1},
\[
 d_1 = d_1(F) = \frac{e(F)}{|F|-1} = \frac{s}{r-1}.
\]
Throughout we assume $r\ge 3$ (otherwise $F$ is an edge and everything is trivial).

Note for later that if $F$ is strictly $1$-balanced then
$F$ is $2$-connected: otherwise, it would be possible to write $F$ as $F_1\cup F_2$,
where $F_1$ and $F_2$ have at least two vertices and overlap in exactly one vertex.
But then
\[
 e(F) = e(F_1)+e(F_2) < d_1(|F_1|-1)+d_1(|F_2|-1) = d_1(|F|-1) = e(F),
\]
a contradiction.

We will prove the analogue of Theorem~\ref{th1} for nice graphs $F$ (see Definition~\ref{defnice}), and
the analogue of Theorem~\ref{th2} for all strictly $1$-balanced $F$, in Theorem~\ref{th-bal} below.
At the end of this section we will use a variation of the method to prove Theorem~\ref{thbal}.
The coupling results are slightly awkward to formulate, since we cannot directly encode copies of $F$
by an $r$-uniform hypergraph.

Let $F$ be a fixed ($t$-)graph with $r$ vertices. By an \emph{$F$-graph} $H_F$ we mean a pair $(V,E)$ where
$V$ is a finite set of vertices and $E$ is a set of distinct copies of $F$ whose
vertices are all contained in $V$. We refer to the copies as \emph{$F$-edges}.
Equivalently, an $F$-graph is an $r$-uniform labelled multi-hypergraph, where each hyperedge $h$ is labelled by one of the $r!/\aut(F)$ possible copies of $F$ on $V(h)$, and we may have two or more hyperedges with the same vertex set as long as they have different labels.

For $n\ge 1$ and $0\le \pi \le1$ we write $H_F(n,\pi)$
for the random $F$-graph with vertex set $[n]$ in which each of the
\[
 M= \binom{n}{r}\frac{r!}{\aut(F)}
\]
possible copies of $F$ (i.e., possible labelled hyperedges)
is present independently with probability $\pi$. Thus, when $F=K_r$,
an $F$-graph is exactly an $r$-uniform hypergraph, and $H_F(n,\pi)=H_r(n,\pi)$.

\begin{theorem}\label{th-bal}
Let $F$ be a fixed strictly $1$-balanced $t$-uniform hypergraph, $t\ge 2$, with $|F|=r\ge 3$ and $e(F)=s$.
Let $d_1=s/(r-1)$.
There are positive constants $\eps$ and $a$ such that if $p=p(n)\le n^{-1/d_1+\eps}$ then,
for some $\pi=\pi(n)\sim ap^s$,
we may couple $G=H_t(n,p)$ and $H_F=H_F(n,\pi)$ such that, with probability $1-o(1)$,
for every $F$-edge present in $H_F$ the corresponding copy of $F$ is present in $G$.
Furthermore, if $F$ is nice, then we may take $a=1$.
\end{theorem}
In other words, in the same one-sided sense as in Theorem~\ref{th1}, and up to a small
change in density, the copies of $F$ in $H_t(n,p)$ (i.e., $G(n,p)$ when $t=2$)
are distributed randomly as if each was present independently.

The slightly awkward statement of Theorem~\ref{th-bal} `does the job' with respect to $F$-factors, for nice $F$, giving Theorem~\ref{thnice} as a corollary.

\begin{proof}[Proof of Theorem~\ref{thnice}]
Fix a nice ($t$-uniform hyper)graph $F$ with $r$ vertices and $s$ edges, and define $p_0$ as in the statement of the theorem, noting that $p_0$ is asymptotically the value of $p$ for which
each vertex of $G(n,p)$ (or $H_t(n,p)$) is on average in $\log n$ copies of $F$.
Then for $p=(1+\gamma) p_0$, say, the $\pi$ in Theorem~\ref{th-bal} satisfies
$\pi\sim p^s=(1+\gamma)^s (\aut(F)/r!) \pi_0$, where $\pi_0$ is defined in Theorem~\ref{Kthm}.
Let $\hsH_F$ be the random simple hypergraph obtained from $H_F$ by replacing each $F$-edge by a hyperedge with the same vertex set (i.e., forgetting the labels) and removing any multiple edges. Then $\hsH_F$
has the distribution of $H_r(n,\pi')$ for
\[
 \pi'=1-(1-\pi)^{r!/\aut(F)} \sim (r!/\aut(F)) \pi \sim (1+\gamma)^s \pi_0,
\]
so we have $\pi'\ge (1+\gamma/2)\pi_0$, say, for $n$ large enough. Thus by Kahn's Theorem~\ref{Kthm}, whp $H_n(r,\pi')$
has a perfect matching. When this holds and the coupling described in Theorem~\ref{th-bal} succeeds,
for each hyperedge in the matching we find some copy of $F$ in $H_t(n,p)$
with the same vertex set, leading to an $F$-factor. The reverse bound is (as is well
known) immediate: if $p=(1-\gamma)p_0$ then whp there will be vertices of $H_t(n,p)$ not in any
copies of $F$.
\end{proof}

To prove Theorem~\ref{th-bal}
we will follow the strategy of the proof of Theorem~\ref{th1} as closely as possible; the main
complication will be in the deterministic part, namely the analogue of Lemma~\ref{l2}.

Given an $F$-graph $H_F$, let $\hH_F$ be the underlying $r$-uniform multi-hypergraph, where we replace
each $F$-edge by a hyperedge formed by the vertex set of $F$ (i.e., forget the labels), and let $G(H_F)$ be the simple
($t$-)graph\footnote{From now on we mostly write just `graph', only occasionally reminding the reader of the case $t\ge 3$.} formed by taking the graph union of the copies of $F$ present as $F$-edges in $H_F$.
We define \emph{avoidable configurations}
in (multi)-hypergraphs as before, now noting that two hyperedges with the same vertex set form
an avoidable configuration (the nullity is $r-1\ge 2$). We say that $H_F$ contains
an avoidable configuration if $\hH_F$ does.

The next deterministic lemma describes how the union of copies of $F$ can create an `extra' copy $F_0$.

\begin{lemma}\label{l2F}
Let $F$ be a $2$-connected ($t$-)graph with $r$ vertices,
let $H_F$ be an $F$-graph, and let $F_0$ be a copy of $F$, not present as an $F$-edge in $H_F$,
such that $F_0\subset G(H_F)$. Then either (i) $\hH_F$ contains an avoidable configuration,
or (ii) $\hH_F$ contains a clean $k$-cycle $H$ for some $2\le k\le e(F)$, with every edge
of $F_0$ contained in some hyperedge in $H$. Furthermore, if $F$ is nice, then (i) holds.
\end{lemma}
\begin{proof}
We may assume that $H_F$ is minimal with the given property. Let its $F$-edge-set be $F_1,\ldots,F_j$,
so these are distinct graphs isomorphic to $F$ whose union contains $F_0$. Let $h_1,\ldots,h_j$
be the corresponding ($r$-element) hyperedges, so $h_i=V(F_i)$, and let $H=\hH_F$,
a (multi-)hypergraph with hyperedges $h_1,\ldots,h_j$.
Since $F$ is connected, it is easy to see that $H$ is connected.
Suppose that $H$ has a pendant hyperedge, i.e., a hyperedge $h$ that meets $H'=H-h$ only in a single
vertex $v$. Then, by minimality of $H$, at least one edge of $F_0$ is included in $h$, and at least
one edge of $F_0$ is included in $H'$. In particular, $F_0$ includes at least one vertex \emph{other than $v$}
in each of $h$ and $H'$. Since $h$ and $H'$ meet only in $v$, it follows that $v$ is a cut-vertex in $F_0$,
contradicting the assumption that $F$ is $2$-connected.

So we may assume that $H$ has no pendant hyperedges.
By minimality, every hyperedge of $H$ contributes at least one edge to $F_0$,
so $j\le e(F)\le \binom{r}{t}\le 2^r$.

If $H$ is complex, then $H$ is an avoidable configuration and we are done.
Suppose not, so in particular $H$ has no repeated hyperedges.
Certainly $e(H)\ge 2$ (since $F_1\ne F_0$), so $H$ cannot be a tree.
Thus $H$ is unicyclic, and in fact it is a clean $k$-cycle
for some $k\ge 2$ (see Figure~\ref{fig_cycles}). Note that $k=e(H)=j\le e(F)$.
This completes the proof of the main statement. It remains only to deduce a contradiction
in the case that $F$ is nice (so $H$ must have been complex after all).

So suppose that $F$ is nice.
Let $C$ be the ($2$-)graph\footnote{$C$ is a graph even if $t\ge 3$.} cycle corresponding to $H$, so each hyperedge $h$ of $H$
consists of two consecutive vertices of $C$ and $r-2$ `external' vertices.
Then $G(H_F)$ cannot contain any ($t$-)edges within $V(C)$ other than the edges of the $2$-graph $C$.
Since $F$ is $3$-connected, it is not a subgraph of $C$, so $F_0$ contains a vertex $v$ outside $C$. 
Assume without loss of generality that $v\in h_1$.
Let $x$ and $y$ be the vertices of $C$ in $h_1$. Then $F_0\subset G(H_F)=F_1\cup \bigcup_{j=2}^k F_j$,
a union of two ($t$-)graphs whose vertex sets intersect in $\{x,y\}$.
Since $F_0$ is 3-connected, it follows that $V(F_0)\subset V(F_1)=h_1$, so in fact these
two sets of $r$ vertices are the same. Now every $F_j$ contributes
at least one edge to $F_0$. For $j>1$ this edge can only be $xy$,
so we conclude that $k=2$ and that $F_0\subset F_1+xy$ (so also $t=2$). Hence it is possible
to transform $F$ into an isomorphic graph by adding one edge and then deleting
one edge. Since $F$ is nice, this is impossible.
\end{proof}

\begin{definition}
Let $H_F$ be an $F$-graph and let $F_1$ be an $F$-edge of $H_F$. We say that
$F_0$ is an \emph{extra copy of $F$ in $H_F$ meeting $F_1$} if

(i) $F_0$ is not present as an $F$-edge in $H_F$,

(ii) all ($t$-)edges of $F_0$ are present in $G(H_F)$, and

(iii) $F_0$ and $F_1$ share at least one ($t$-)edge.

\noindent
We write $N_F(H_F,F_1)$ for the number of extra copies of $F$ in $H_F$ meeting $F_1$.
\end{definition}

The first two conditions above express that when we take the union of the copies of $F$ encoded
by $H_F$, then $F_0$ appears as an `extra' copy of $F$.

\begin{definition}
Let $M_F$ denote the supremum
of $N_F(H_F,F_1)$ over all $F$-graphs $H_F$ and $F_1\in E(H_F)$, where
$H_F$ contains no avoidable configuration.
\end{definition}

In this notation, Lemma~\ref{l2} says that for $r\ge 4$, $M_{K_r}=0$. Similarly,
Lemma~\ref{l2F} has the following corollary.

\begin{corollary}\label{cMF}
If $F$ is $2$-connected, then $M_F$ is finite. If $F$ is nice, then $M_F=0$.
\end{corollary}
\begin{proof}
The second statement is immediate from Lemma~\ref{l2F} and the definition of $M_F$.
For the first, let $F_1$ be an $F$-edge of an
$F$-graph $H_F$ containing no avoidable configuration, and let $F_0$ be an extra
copy of $F$ in $H_F$ meeting $F_1$.
Then, by Lemma~\ref{l2F}, the hypergraph $\hH_F$ contains
a clean $k$-cycle $H$ for some $2\le k\le e(F)$, with each ($t$-)edge of $F_0$
contained in a hyperedge of $H$. Consider the hyperedge $h=V(F_1)$ corresponding to $F_1$.
Then $F_0$ and $F_1$ share an edge $e$, which must be contained
in some hyperedge in $H$. So $h$ shares at least two vertices with $H$. If $h$ is not already
a hyperedge of $H$, it follows that $H\cup \{h\}\subset \hH_F$ is complex
and thus an avoidable configuration,
contradicting our assumptions. Hence $h$ is indeed a hyperedge of $H$.

For any extra copy $F_0$ meeting $F_1$ we obtain a (unicyclic) witness $H$ as above.
Each $H$ can be a witness
for at most $O(1)$ copies $F_0$, since $H$ has $O(1)$ vertices and so contains
$O(1)$ subgraphs isomorphic to $F$. On the other hand, if $\hH_F$ contains two distinct witnesses
then, since they share a hyperedge, their union is complex and has at most $2k\le 2e(F)\le 2^{r+1}$ hyperedges,
again contradicting our assumption that $H_F$ contains no avoidable configuration.
\end{proof}

We are now ready to prove Theorem~\ref{th-bal}.
\begin{proof}[Proof of Theorem~\ref{th-bal}]
We follow the proof of Theorem~\ref{th1} (for the case $F$ nice) or Theorem~\ref{th2}
as closely as possible. In particular, we follow Algorithm~\ref{alg} \emph{mutatis mutandis},
testing copies of $F$ for their presence in $G(n,p)$ (or $H_t(n,p)$ for the $t$-graph case)
and simultaneously constructing
a random $F$-graph $H_F$ with the distribution of $H_F(n,\pi)$.

As before, it is convenient to assume that $p\le n^{-1/d_1+o(1)}$. 
Then $\pi\le n^{-(r-1)+o(1)}$, so the expected degrees in $H_F(n,\pi)$ or its underlying hypergraph are at most $n^{o(1)}$.
Writing $G(H_F)$ for the ($t$-)graph associated to $H_F$,
it follows as before that, for some $\Delta=n^{o(1)}$, the event $\cB_1$ that any vertex has degree more 
than $\Delta$ in $G(H_F)$ has probability $o(1)$. Furthermore, by (a trivial modification of) Lemma~\ref{l1}, the event
$\cB_2$ that $H_F$ contains an avoidable configuration has probability $o(1)$.

The core of the argument is exactly as before: we test the edge-sets $E_1,\ldots,E_M$
of the possible copies $F_1,\ldots,F_M$ for their presence
in $G=G(n,p)$, or $G=H_t(n,p)$ for $t\ge 3$, one-by-one. Our coupling only fails if the conditional probability $\pi_j$
that the $j$-th test succeeds is smaller than $\pi$, and the corresponding $F$-edge
$F_j$ is present in the random $F$-graph $H_F$ that we are constructing; we write
$\cF_j$ for this latter event. We aim to show that in this case, $\cB_1\cup\cB_2$ holds.
We argue by contradiction, assuming that $\cF_j$ holds, but neither $\cB_1$ nor $\cB_2$
does; our aim is to show that then $\pi_j\ge \pi$.
Note that under these assumptions,
$R\cup E_j\subset G(H_F)$ and so $\Delta(R\cup E_j)\le \Delta$.

The derivation of \eqref{key} did not use \emph{any}
properties of the $E_i$, except to bound $|E_j\setminus R|$ by $|E_j|=e(K_r^{(t)})$
in the last step. Thus we have
\[
 \pi_j \ge p^s (1-Q_j)
\]
with $Q_j$ defined as in \eqref{Qdef}, as before. For $i\in N_1=N_{j,1}$, we let $S$ be the ($t$-)graph on $V(F_i)$ formed by all ($t$-)edges
in $E_i=E(F_i)$ that are also contained in $R\cup E_j$,
and write $e_i=|E_i\setminus E(S)|=|E_i|-e(S)$; thus the contribution from this $i\in N_1$
to $Q_j$ is precisely $p^{e_i}$.

We split the contribution to $Q_j$ into two types, according to whether $e_i>0$ or not, writing
\[
 Q_j = \tQ_j + |D_j|
\]
where, as before,
\[
 D_j=\{i\in N_{j,1}: E_i\subset R\cup E_j\}.
\]
As before, we can split the sum $\tQ_j$ according to the number $k+1$ of components 
and number $m$ of edges of $S$, a non-trivial subgraph of $F$. Since we assume $\Delta(R\cup E_j)\le \Delta$,
we obtain as before (see \eqref{choices}) that
\begin{equation}\label{tQj}
 \tQ_j \le \sum_{k,m} r n^k ((t-1)\Delta)^{r-1-k} p^{s-m}.
\end{equation}
Suppose $S$ has $k+1\ge 2$ components, with $r_1,\ldots,r_{k+1}$ vertices
and $s_1,\ldots,s_{k+1}$ edges, respectively.
Each component is a subgraph of $F_i$, which is $1$-balanced, so
\begin{equation}\label{mbd}
 m = e(S) = \sum_\ell s_\ell \le \sum_\ell d_1(r_\ell-1) = d_1(r-1-k) = s-d_1 k.
\end{equation}
In fact, $F$ is strictly $1$-balanced, so we have a strict inequality if any $r_i$ is
in the range $2\le r_i\le r-1$. As before, $S$ contains at least one edge, so we cannot have
all $r_i$ equal to one. Thus, when $S$ is disconnected, i.e., $k\ge 1$, we have
a strict inequality in \eqref{mbd}. When $k=0$, by the way we split the sum $Q_j$ we have $e_i=s-e(S)>0$,
so we have a strict inequality. It follows that all terms in the sum in \eqref{tQj}
are at most
\[
 n^{k+o(1)}p^{d_1 k +1} \le n^{-1/d_1+o(1)}=o(1),
\]
so $\tQ_j=o(1)$.

Turning to $|D_j|$, if $i\in D_j$ then the $F$-edge $F_i$ is not present in $H_F$ (since $i\in N$, so
by the definition of the algorithm we did not include $F_i$ as an $F$-edge of $H_F$).
On the other hand, since $\cF_j$ holds, the $F$-edge corresponding to $F_j$
is present, and $R\cup E_j\subset G(H_F)$. Thus $F_i$ is an extra copy of $F$ in $H_F$ which (by definition
of $N_{j,1}$) meets $F_j$. We assume $\cB_2$ does not hold, so the number of possible such $i$ is at most $M_F$. In conclusion,
\[
 Q_j \le o(1) + M_F.
\]

If $F$ is nice then $M_F=0$ by Corollary~\ref{cMF}, so $Q_j=o(1)$ and we are done.
For general strictly $1$-balanced $F$, we know that $F$ is $2$-connected, so $M_F$ is finite
by Corollary~\ref{cMF}.
Thus we may bound $Q_j$ by $C=M_F+1$, say. Now we let $c=1/(2C)$ and introduce
extra tests (one per copy of $F$) as in the proof of Theorem~\ref{th2}.
In this case we have $\pi_j\ge cp^s(1-cQ_j)\ge cp^s(1-1/2) = ap^s$, for $a=c/2$,
so (if neither $\cB_1$ nor $\cB_2$ holds), we have $\pi_j\ge \pi$, as required.
\end{proof}

A slight variant of the proof above, with almost identical arguments but different
parameters, yields Theorem~\ref{thbal}.

\begin{proof}[Proof of Theorem~\ref{thbal}]
Fix $F$ which is $1$-balanced, but need not be strictly $1$-balanced. Define $r=|F|$, $s=e(F)$
and $d_1=d_1(F)$ as before. Pick a constant $a$ such that $d_1a>2$, and set
\[
 p = (\log n)^a n^{-1/d_1} \text{\qquad and\qquad} \pi=C (\log n) n^{-(r-1)},
\]
where the constant $C$ is chosen large enough that the random $F$-graph $H_F=H_F(n,\pi)$ (or rather, its
underlying hypergraph), whp contains a perfect matching;
such a constant exists by the result of Johansson, Kahn and Vu~\cite{JKV}, and Theorem~\ref{Kthm} gives an explicit value.
Note that we may write $\pi = c p^s/2$ where
\[
 c=\Theta((\log n)^{1-as}) = \Theta((\log n)^{1-ad_1(r-1)}).
\]

We follow the proof of Theorem~\ref{th-bal} above, in particular in the form with additional 
tests with probability $c$ as in the proof of Theorem~\ref{th2}. Since the expected degrees in $H_F$ are
of order $\log n$, we may take $\Delta=O(\log n)$. As before, the event $\cB_1$ that $G(H_F)$
has maximum degree more than $\Delta$, and the event $\cB_2$ that $H_F$ contains
an avoidable configuration, have probability $o(1)$. To complete the proof we need only show
that when neither $\cB_1$ nor $\cB_2$ holds, but the $F$-edge corresponding to $F_j$ is included
in $H_F$, then $\pi_j\ge \pi$. As before, we have $\pi_j\ge cp^s(1-cQ_j)$,
so it suffices to show that in this case $Q_j\le 1/(2c)$.

Estimating $Q_j$ as in the proof of Theorem~\ref{th-bal}, but this time not separating out the $e_i=0$
term, we have
\[
 Q_j \le \sum_{k,m} r n^k ((t-1)\Delta)^{r-1-k} p^{s-m} \le \sum_{k=0}^{r-2} r n^k ((t-1)\Delta)^{r-1-k} p^{d_1k},
\]
where we used \eqref{mbd} (whose derivation only assumed that $F$ is $1$-balanced) to bound $s-m$ by $d_1k$,
and note as usual that the overlap graph $S$ contains at least one edge, so the number
$k+1$ of components is at most $r-t+1\le r-1$.
Now $np^{d_1}=(\log n)^{ad_1}$, while $\Delta=O(\log n)$. It follows easily that
the term $k=r-2$ dominates the sum above, so $Q_j=O((\log n)^{ad_1(r-2)+1})$. Since $ad_1>2$,
it follows that $cQ_j=o(1)$, completing the proof.
\end{proof}

Gerke and McDowell~\cite{GMcD} consider multipartite multigraph analogues of the $t=2$ case of Theorem~\ref{thbal}.
Their main focus is the `nonvertex-balanced' case, but they also prove a version for arbitrary $F$
losing a factor $n^{o(1)}$ in the edge probability. The proof above extends \emph{mutatis mutandis}
to reduce this factor to $(\log n)^{O(1)}$ when $F$ is $1$-balanced. Since this is not our main
focus, we only outline the details.

Let $F$ be a multigraph, which we will view as a graph with a positive integer weight on each edge.
Let $V(F)=\{v_1,\ldots,v_r\}$. As in \cite{GMcD}, we will look for an $F$-factor in a random
graph $G$ where we first divide the vertices of $G$ into $r$ equally sized disjoint sets $V_1,\ldots,V_r$,
and only consider copies of $F$ with each $v_i$ mapped to a vertex in $V_i$. In~\cite{GMcD}, $G$
is a random multigraph, but their formulation is exactly equivalent to the following: we take
all edges of $G$ to be present independently, and an edge between $V_i$ and $V_j$ has probability $p^{m(ij)}$,
where $m(ij)$ is the multiplicity of $v_iv_j$ in $F$. Then we look for a copy
(restricted as above) of the simple graph
underlying $F$ in this random graph $G$.

The coupling arguments above translate immediately to this setting: we are still working in a product
probability space, and if $E$ is a set of possible edges of $G$, the probability that all
are present is $p^{|E|}$ where now we count edges according to their multiplicity. Nothing else
in the argument needs changing, except the hypergraph input. Let $H$ be the random
$r$-partite $r$-graph $H_r(n,n,\ldots,n,\pi)$ where each of the $n^r$ possible hyperedges is present
independently with probability $\pi$. Then we need to know that if $\pi$ is at least some constant
times $(\log n)n^{-(r-1)}$, then whp $H$ has a complete matching. This statement follows easily
from the Johansson--Kahn--Vu argument as presented by Frieze and Karo\'nski~\cite{FK}, simply
starting with a complete multipartite hypergraph and removing edges one-by-one, rather than
starting with a complete hypergraph. This result also follows from Corollary 1.2 of
Bal and Frieze~\cite{BF}, itself a consequence of a more general result needed there.

\section{Open questions}\label{sec_q}

The motivation for this paper was to understand, in the Johansson--Kahn--Vu context in particular, the relationship
between the distribution of copies of $K_r$ in $G(n,p)$ and the random hypergraph $H_r(n,\pi)$,
$\pi\sim p^{\binom{r}{2}}$.
This rather vague question seems to make sense much more generally.
The method used here works for $p$ up to $n^{-2/r+\eps}$ for some $\eps>0$. How large is this $\eps$?
More interestingly, up to what $p$ is a result analogous to Theorem~\ref{th1} true?
It should break down when a typical edge of $G(n,p)$ has a significant probability
of being in a copy of $K_r$, since then a significant fraction of the $K_r$s in $G(n,p)$ share
edges, and these overlapping pairs are more likely in $G(n,p)$ than in $H_r(n,\pi)$.
Of course, this doesn't rule out some other interesting relationship between $G(n,p)$ and $H_r(n,\pi)$
for even larger $p$.

Turning to general graphs $F$ in place of $K_r$, and looking for sharp results, say, with
$\pi\sim p^{e(F)}$,
one might ask what the right class of graphs $F$ is. The conditions
in Theorem~\ref{th-bal}, and thus Theorem~\ref{thnice}, are what makes the proof work, and are presumably more restrictive than needed.
Strictly $1$-balanced is a natural assumption, but even without this assumption
there might still be a sensible
way to relate copies of $F$ in $G(n,p)$ to a suitable hypergraph, which might or might not be $H_r(n,\pi)$,
depending on $F$ and on the value of $p$.

In Theorem~\ref{th1}, one could ask how large a failure probability must be allowed
in the coupling. The proof as given yields $n^{-\delta}$ for some positive $\delta$,
coming from the probability that $H_r(n,\pi)$ contains an avoidable configuration.
But it could be that much smaller failure probabilities are possible. Also, what about comparing
the distributions in some different way? In particular, looking for some two-sided sense in which they
are close?

Finally, it would be interesting to know whether the simple proof of Theorem~\ref{th2} given (in a very
slightly different setting) by Kim~\cite{Kim2003}, and outlined in Section 4.1
of the draft arXiv:1802.01948v1
of the present paper, can be extended to $r\ge 4$, perhaps by some kind of induction. It's not
at all clear whether this is possible, though.

\begin{ack}
Part of this work was carried out while the author was a visitor at IMPA in Rio de Janeiro;
the author is grateful to Rob Morris and to IMPA for their hospitality. The author would like to thank
Annika Heckel for a helpful discussion of the hypergraph case, the referees for helpful suggestions, and 
the editors of \emph{Random Structures and Algorithms} for their patience with the very late revision.
\end{ack}


\begin{thebibliography}{9}

\bibitem{AY} N. Alon and R. Yuster,
 Threshold functions for $H$-factors,
 \emph{Combin. Probab. Comput.} {\bf 2} (1993), 137--144.

\bibitem{BF} D.~Bal and A.~Frieze,
 Rainbow matchings and Hamilton cycles in random graphs,
 \emph{Random Struct. Alg.} {\bf 48} (2016), 503--523.

\bibitem{rmax} B.~Bollob\'as and O.~Riordan,
 Constrained graph processes, \emph{Electronic Journal of Combinatorics} {\bf 7} (2000),
\#R18 (electronic, 20 pp.)

\bibitem{Erdos} P. Erd\H os,
 On the combinatorial problems which I would most like to see solved,
 \emph{Combinatorica} {\bf 1} (1981), 25--42

\bibitem{FK} A.~Frieze and M.~Karo\'nski,
 \emph{Introduction to Random Graphs}, Academic Press (2016), 478 pp.

\bibitem{GMcD} S.~Gerke and A.~McDowell,
 Nonvertex-balanced factors in random graphs,
 \emph{J. Graph Theory} {\bf 78} (2015), 269--286.

\bibitem{Heckel} A.~Heckel,
 Random triangles in random graphs, 
 \emph{Random Struct. Alg.} {\bf 59} (2021), 616--621. 

\bibitem{JKV} A. Johansson, J. Kahn and V. Vu,
 Factors in random graphs, \emph{Random Struct. Alg.} {\bf 33} (2008), 1--28.

\bibitem{Kahn1} J. Kahn,
 Asymptotics for Shamir's Problem, preprint (2019), arXiv:1909.06834v1

\bibitem{Kahn2} J. Kahn,
 Hitting times for Shamir's problem,
 \emph{Trans. Amer. Math. Soc.} {\bf 375} (2022), 627--668.
 
\bibitem{Kim2003} J. H. Kim,
 Perfect matchings in random uniform hypergraphs,
 \emph{Random Struct. Alg.} {\bf 23} (2003), 111--132.

\bibitem{Kriv97} M. Krivelevich,
 Triangle factors in random graphs,
 \emph{Combin. Probab. Comput.} {\bf 6} (1997), 337--347.

\bibitem{RW-Janson} O.~Riordan and L.~Warnke,
 The Janson inequalities for general up-sets,
 \emph{Random Struct. Alg.} {\bf 46} (2015), 391--395.

\bibitem{Rucinski} A. Ruci\'nski,
 Matching and covering the vertices of a random graph by copies of a given graph,
 \emph{Discrete Math} {\bf 105} (1992), 185--197.

\bibitem{SS} J. Schmidt and E. Shamir,
 A threshold for perfect matchings in random d-pure hypergraphs,
 \emph{Discrete Math} {\bf 45} (1983), 287--295. 


 

\end{thebibliography}
\end{document}